\documentclass{amsart}
\usepackage{amsthm,amsfonts,amsmath,amscd,amssymb,latexsym,epsfig}

\newcommand{\cx}{{\mathbb C}}

\newcommand{\tr}{\operatorname{tr}}

\newcommand{\im}{\operatorname{Im}}

\newcommand{\End}{\operatorname{End}}

\newcommand{\Ker}{\operatorname{Ker}}

\newcommand{\Mat}{\operatorname{Mat}}

\numberwithin{equation}{section}

\newtheorem{theorem}{Theorem}[section]

\newtheorem{lemma}[theorem]{Lemma}

\newtheorem{corollary}[theorem]{Corollary}

\newtheorem{proposition}[theorem]{Proposition}

\theoremstyle{remark}

\newtheorem{remark}[theorem]{Remark}

\newtheorem{definition}[theorem]{Definition}

\newtheorem{example}[theorem]{Example}

\newcommand{\oC}{{\mathbb{C}}}

\newcommand{\oP}{{\mathbb{P}}}

\newcommand{\oR}{{\mathbb{R}}}

\newcommand{\sJ}{{\mathcal{J}}}

\newcommand{\sO}{{\mathcal{O}}}

\newcommand{\sV}{{\mathcal{V}}}

\begin{document}

\title{Jumps, folds and  hypercomplex structures}
\author{Roger Bielawski* \and Carolin Peternell}
\address{Institut f\"ur Differentialgeometrie,
Leibniz Universit\"at Hannover,
Welfengarten 1, 30167 Hannover, Germany}
\email{bielawski@math.uni-hannover.de}

\thanks{Both authors are members of, and the second author is fully supported by the DFG Priority
Programme 2026 ``Geometry at infinity".}


\begin{abstract} 
We investigate the geometry of the Kodaira moduli space $M$ of sections of $\pi:Z\to \oP^1$, the normal bundle of which is allowed to jump from $\sO(1)^{n}$ to $\sO(1)^{n-2m}\oplus\sO(2)^{m}\oplus\sO^{m}$. In particular,  we identify the natural assumptions which guarantee that the Obata connection of the hypercomplex part of $M$ extends to a logarithmic connection on $M$.
\end{abstract}


\maketitle

\thispagestyle{empty}

\section{Introduction}

It is well known that a hyperk\"ahler or a hypercomplex structure on a smooth manifold $M$ can be encoded in the {\em twistor space}, which is a complex manifold $Z$ fibring over $\oP^1$ and equipped with an antiholomorphic involution $\sigma$ covering the antipodal map. The manifold $M$ is recovered as the parameter space of $\sigma$-invariant sections with normal bundle isomorphic to $\sO(1)^{\oplus n}$ ($n=\dim_\cx M$). If we start with an arbitrary                                                                                                                                                                                                                                                                                                                complex manifold $Z$ equipped with a holomorphic submersion $\pi:Z\to \oP^1$ and an involution $\sigma$, then the corresponding component  of the Kodaira moduli space of sections of $\pi$ will typically also contain sections with other normal bundles $\bigoplus_{i=1}^n\sO(k_i)$. This kind of jumping normal bundle  attracted recently attention in the case of $4$-dimensional hyperk\"ahler manifolds \cite{Hit2,Biq,Dun}, in the context of a phenomenon known as {\em folding} (one speaks then of {\em folded hyperk\"ahler metrics}).
\par
Folded hyperk\"ahler structures do not exhaust all geometric possibilities which arise when the normal bundle is allowed to jump. Even in four dimensions there are examples which are not folded (Example \ref{proj} below). The aim of this paper is to investigate the natural extension of the hypercomplex geometry arising on such manifolds of sections (folded or not). More precisely, we are interested in the differential geometry of the (smooth) parameter space $ M$ of sections of $\pi:Z\to \oP^1$ with normal bundle $N$ isomorphic $\bigoplus_{i=1}^n\sO(k_i)$, where  each $k_i\geq 0$. We shall discuss only the purely holomorphic case, i.e. we are interested  in all sections, not just $\sigma$-invariant. Choosing an appropriate $\sigma$ allows one to carry over all results to  hypercomplex or split hypercomplex manifolds.
\par
Our particular object of interest is the (holomorphic) {\em Obata connection} $\nabla$, i.e. the unique torsion free connection preserving the hypercomplex (or, rather, the biquaternionic, i.e. complexified hypercomplex) structure. This is defined on the open subset $U$ of $M$ corresponding to the sections with normal bundle isomorphic to $\sO(1)^{\oplus n}$. The general twistor machinery (see, e.g. \cite{BE}) implies that $\nabla$ extends to a first order differential operator $D$ on sections of certain vector bundle defined over all of $M$. 
Our point of view is to regard $D$ as a particular type of meromorphic connection with polar set $\Delta=M\backslash U$. In general,  this {\em meromorphic Obata connection}  can have higher order poles along $\Delta$. We show, however, that in the case when $M$ arises from a (partial) compactification of the twistor space of a hypercomplex manifold, the meromorphic Obata connection has a simple pole, and in fact it is then a logarithmic connection.

\section{Geometry of jumps}

\subsection{Two examples}
We begin with two examples illustrating the different geometric possibilities occuring when the normal bundle of a twistor line jumps. The first example is the basic example of a {\em hyperk\"ahler fold}, as explained by Hitchin \cite{Hit2}.

\begin{example} (Calabi-Eguchi-Hanson gravitational instanton)  The twistor space $Z$ of the Calabi-Eguchi-Hanson metric is the resolution of the variety
$$\{(x,y,z)\in \sO(2)\oplus\sO(2)\oplus\sO(2); xy=(z-p_1)(z-p_2)\},$$
where $p_i(\zeta)=a_i\zeta^2+2b_i\zeta+c_i$, $i=1,2$, are two fixed sections of $\sO(2)$. The sections of the projection $Z\to \oP^1$ can be described as follows \cite{Hit1}: let $z=a\zeta^2+2b\zeta+c$ be a section of $\sO(2)$ and let $\alpha_i,\beta_i$ be the roots of $z-p_i$, $i=1,2$. Then the sections are given by
$$  z=a\zeta^2+2b\zeta+c, \enskip x=A(\zeta-\alpha_1)(\zeta-\alpha_2),\enskip y=B(\zeta-\beta_1)(\zeta-\beta_2),$$
where $AB=(a-a_1)(a-a_2)$.
A computation by Hitchin in \cite{Hit1} determines the splitting type of the normal bundle and can be interpreted as follows. Elements $\tau$ of $GL_2(\cx)$ with $\det\tau=-1$ and $\tr\tau =0$ satisfy $\tau^2=1$. For any pair $p_1,p_2$ of quadratic polynomials there exists such an $\tau$ exchanging $p_1$ and $p_2$, which, consequently,  acts on $Z$.
The normal bundle of a $\tau$-invariant section splits as $\sO(2)\oplus \sO$; otherwise as  $\sO(1)\oplus\sO(1)$.
\par
To see this directly, observe that modulo translations and the action of $GL_2(\cx)$, $p_1(\zeta)=\zeta$ and $p_2(\zeta)=-\zeta$. The involution $\tau$ is then simply $\zeta\mapsto -\zeta$, and since the normal bundle $N$ of an invariant section satisfies $\tau^\ast N=N$, it must split into line bundles of even degrees. The $\tau$-invariant sections of $Z$ are given by
$z=a\zeta^2 +c$ and by
$$ x=A\left(\zeta+\frac{1+\sqrt{ac}}{a}\right)\left(\zeta-\frac{1-\sqrt{ac}}{a}\right), \enskip y=B\left(\zeta-\frac{1+\sqrt{ac}}{a}\right)\left(\zeta +\frac{1-\sqrt{ac}}{a}\right),$$
where $AB=a^2$. Consequently, for every $\zeta\neq 0,\infty$, the map given by intersecting a section with the fibre $\pi^{-1}(\zeta)$ remains surjective when restricted to sections with normal bundle $\sO(2)\oplus \sO$.
\label{CEH}\end{example}

\begin{example} ($\oR \rm P^4$) Let $Z=\oP\bigl(\sO(1)\oplus\sO(1)\oplus\sO\bigr)$ be the compactification of the twistor space of $\oR^4$. Sections are described as projective equivalence classes:
$$\zeta\mapsto [a_1\zeta+b_1,a_2\zeta+b_2,c],\enskip a_i,b_i,c\in \cx,\enskip c\neq 0\implies \det\begin{pmatrix} a_1 & b_1\\a_2 & b_2\end{pmatrix}\neq 0.$$
The normal bundle of such a section is $\sO(1)\oplus\sO(1)$ if $c\neq 0$, and $\sO(2)\oplus \sO$ if $c=0$. The manifold of all sections is an open subset of $\cx\rm P^4$, and the manifold of real sections, i.e. satisfying $b_1=-\bar a_2$, $b_2=\bar a_1$, $c=\bar c$, is $\oR \rm P^4$.  
\par
Observe that the twistor lines with normal bundle $\sO(2)\oplus \sO$ are all contained in their own minitwistor space $\oP\bigl(\sO(1)\oplus\sO(1)\oplus 0\bigr)\simeq \oP^1\times \oP^1$.
\label{proj}\end{example}

From the point of view of hyperk\"ahler geometry, the difference between the two examples is clear: in the first case, there is a well defined $\sO(2)$-valued symplectic form along the fibres of $Z$. In the second example, this is not the case.  The second example does not fit into Hitchin's theory of folded hyperk\"ahler manifolds: the $3$-dimensional submanifold of real twistor lines with normal bundle $\sO(2)\oplus \sO$ is $\oR \rm P^3$, so it is not even a contact manifold.
\par
Our aim now is to investigate both the common features and the differences in the behaviour of the hypercomplex structure and of the Levi-Civita (i.e. Obata) connection.

\subsection{$2$-Kronecker structures\label{Kr}}
Let $Z$ be a complex manifold of dimension $n+1$ and $\pi:Z\to \oP^1$ a surjective holomorphic submersion. We are interested in the (necessarily smooth) parameter space $M$ of sections of $\pi$ with normal bundle $N$ isomorphic to $\bigoplus_{i=1}^n\sO(k_i)$, where $k_i\in\{0,1,2\}$ and $n=\sum k_i$. Its dimension (as long as it is nonempty) is $2n$ and we consider a connected component $M$ which contains a section with normal bundle isomorphic to $\sO(1)^{\oplus n}$. 
\par
The tangent space $T_m M$ at any $m\in M$ is canonically isomorphic to $H^0(s_m,N)$, where $s_m$ is the section corresponding to $m$. Similarly, we have a rank $n$ bundle $E$ over $M$, the fibre of which is $H^0(s_m,N(-1))$, where $N(-1)=N\otimes \pi^\ast\sO_{\oP^1}(-1)$.  The multiplication  map $H^0(N(-1))\otimes H^0(\sO(1))\to H^0(N)$ induces a homomorphism
$$ \alpha:E\otimes \cx^2\to TM,$$
which is an isomorphism at any $m$ with $N_{s_m/Z}\simeq \sO(1)^n$. It follows that the subset of $M$ consisting of sections with other normal bundles is a divisor $\Delta$ in $M$. We shall assume throughout that the set of singular points of $\Delta$ has codimension $2$ in $M$ (in particular $\Delta$ is reduced). This means that the normal bundle of a section corresponding to a smooth point of $\Delta$ is isomorphic to $\sO(1)^{n-2}\oplus \sO(2)\oplus\sO$. 
\par
Observe also that $\alpha$ is injective on each subbundle of the form $E\otimes v$, where $v$ is a fixed nonzero vector in $\cx^2$. The image $D_v$ of the subbundle $E\otimes v$ is an integrable distribution on $TM$ (sections of $\pi$ vanishing at the zero of $v\in H^0(\sO(1))$) and we recover $Z$ as the space of leaves of the distribution $D$ on $M\times \oP^1$ given by $D|_{M\times [v]}=D_v$.
\begin{remark} We can also define $E$ as the kernel of the evaluation map $H^0(N)\otimes\sO_{\oP^1}\to N$ (which is what we do in \cite{BP1}), i.e.
$$0\to E_m\otimes \sO_{\oP^1}(-1) \stackrel{A}{\longrightarrow} H^0(N)\otimes\sO_{\oP^1}\longrightarrow N\to 0.$$
We obtain again a map $\alpha:E\otimes \cx^2\to TM$ by restricting $A$ to each subspace of the form $E\otimes\langle v\rangle$, $v\in \cx^2$. But then the above sequence identifies $H^0(N(-1))$ with $E_m\otimes H^1(\sO_{\oP^1}(-2)$. Thus, viewing $\alpha$ as the multiplication map $H^0(N(-1))\otimes H^0(\sO(1))\to H^0(N)$ means that we have implicitly identified  $H^1(\sO_{\oP^1}(-2))$ with 
$\cx$. Such an identification yields also a choice of a symplectic form on $H^0(\sO(1))$ within its conformal class, i.e. an identification of $\cx^2$ with $(\cx^2)^\ast$.
\label{subtle}\end{remark}
This geometric structure on $M$ was introduced in \cite{BP1} as an {\em integrable $2$-Kronecker structure}. We now want to present a different point of view, directly in terms of the tangent bundle of $M$.
\par
Let $M$ be a complex manifold  and let $\Delta$ be  a divisor satisfying the above smoothness assumption. Suppose that we are given a codimension $1$ distribution $\sV$ on the smooth locus $\Delta_{\rm reg}$ of $\Delta$. We define $TM(-\sV)$ to be the sheaf of germs of holomorphic vector fields $X$ on $M$ such that $X_x\in \sV_x$ for any $x\in \Delta_{\rm reg}$. If the sheaf $TM(-\sV)$ is locally free, i.e. a vector bundle $F$, then we obtain a homomorphism $\alpha:F\to TM$ from the inclusion $TM(-\sV)\hookrightarrow TM$ (and $\sV=\im\alpha$). In the case of $M$ arising as the parameter space of sections as at the beginning of the subsection, $F\simeq E\otimes \cx^2$ and the action of $\Mat_{2}(\cx)$ gives an  action of complexified quaternions on $TM(-\sV)$. Moreover, for any $\sJ\in SL_2(\cx)$ with  $\tr \sJ=0$ (which implies $\sJ^2=-1$), the $i$-eigensubsheaf of $TM(-\sV)$ is closed under the Lie bracket.
Restricting to a real submanifold of $M$ (and corresponding real slices of $\Delta$ and $\sV$) describes the extension of the hypercomplex or split-hypercomplex geometry to manifolds of sections with jumping normal bundles.

\subsection{Logarithmic hypercomplex structures}
In the setting of the above paragraph, the case of particular interest is $\sV=T\Delta_{\rm reg}$. Vector fields in $TM(-\sV)$ are called then {\em logarithmic} and $TM(-\sV)$ is denoted by $TM(-\log \Delta)$ \cite{Saito}. Another way to characterise logarithmic vector fields is via the condition $X.z\in (z)$, where $z=0$ is the local equation of $\Delta$. This shows, in particular, that the subsheaf $TM(-\log \Delta)$ is closed under the Lie bracket.
\begin{definition} Let $M$ be a complex manifold and $\Delta$ a divisor in $M$ such that its set of singular points is of codimension $2$ in $M$. A {\em logarithmic biquaternionic structure} on $M$ is an action of $\Mat_2(\cx)$ on  $TM(-\log \Delta)$ such that the Nijenhuis tensor of each $A\in \Mat_2(\cx)$ vanishes.\label{log}\end{definition}
\begin{remark} The same definition can be used for real manifolds and we can speak of logarithmic hypercomplex or logarithmic split hypercomplex structures.\end{remark}

Observe that for a logarithmic biquaternionic structure the leaves of the distribution $D_v=\alpha(E\otimes v)$ on $\Delta$ are contained in $\Delta$, i.e. the image of $\Delta$ in each fibre of the twistor space has codimension $1$. In other words $Z$ is a (partial) compactification of the twistor space of a hypercomplex manifold.
More precisely:
\begin{proposition} The following two conditions are equivalent:
\begin{itemize}
\item[(i)] $\im\alpha_x=T_x\Delta$ for each $x\in \Delta_{\rm reg}$;
\item[(ii)] for each $\zeta\in\oP^1$, the map $\Delta\to \pi^{-1}(\zeta)$, given by intersecting a section with the fibre, maps a neighbourhood of each point $x\in \Delta_{\rm reg}$ onto  an $(n-1)$-dimensional submanifold.
\end{itemize}
\label{Tw}\end{proposition}
\begin{proof} Let $f$ denote the map $M\to \pi^{-1}(\zeta)$, given by intersecting a section with the fibre. For an $x\in\Delta_{\rm reg}$, we have 
\begin{equation} \im\alpha_x=H^0(\sO(1)^{n-2}\oplus \sO(2))\subset H^0(N)\simeq T_xM.\label{sD}\end{equation}
Thus $df(\im\alpha_x)$ is an $n-1$-dimensional subspace for any  $x\in \Delta_{\rm reg}$.
Since $f$ is a submersion at $x$, the condition $\im\alpha_x=T_x\Delta$ implies now that the  $f(\Delta_{\rm reg})$  is an immersed $(n-1)$-dimensional submanifold. Conversely, suppose that the condition (ii) holds. Then the image $f(U)$ of a neighbourhood $U$ of $x\in \Delta_{\rm reg}$ is a codimension $1$ submanifold $Z_0$ of $Z$. It follows that $T_x U\simeq H^0(N_{s/Z_0})$, where $s$ is the section corresponding to $x$. Suppose that $N_{s/Z_0}\simeq \bigoplus_{i=1}^{n-1}\sO(k_i)$. Given the injection $N_{s/Z_0}\hookrightarrow N_{s/Z}$, we have (after reordering the $k_i$) $k_1\leq 2$ and $k_2,\dots,k_{n-1}\leq 1$. Since $H^1(N_{s/Z_0})=0$, we have $\dim \Delta_{\rm reg}=h^0(N_{s/Z_0})$ and therefore $\sum_{i=1}^{n-1}(k_i+1)=2n-1$. Thus $(2-k_1)+\sum_{i=2}^{n-1}(1-k_i)=0$ and since each summand is nonnegative, we conclude that $N_{s/Z_0}\simeq \sO(2)\oplus \sO(1)^{n-2}$.
Thus $T_{x}\Delta=\im\alpha_{x}$.
\end{proof}

\begin{remark} This is precisely the situation in Example \ref{proj}. In Example \ref{CEH}, the $2$-Kronecker structure is not logarithmic.
\end{remark}

\section{The meromorphic Obata connection}

The Obata connection of a hypercomplex manifold is the unique torsion-free connection with respect to which the hypercomplex structure is parallel. In the case of a hyperk\"ahler manifold, it coincides with the Levi-Civita connection. From the twistor point of view it is obtained via the Ward transform \cite{Ward,HM}. We now wish to discuss an extension of the Obata connection to a $2$-Kronecker manifold.
\par
Let $Z,M,\Delta,E$ and $\alpha$ be all as in the previous section. We consider the double fibration
$$ M\stackrel{\tau}{\longleftarrow} M\times \oP^1 \stackrel{\eta}{\longrightarrow} Z.$$
The normal bundle $N$ of any section of $\pi$ is isomorphic to the vertical tangent bundle $T_\pi Z=\Ker d\pi$ restricted to the section, and, consequently, the (holomorphic) tangent bundle $TM$ can be viewed as the Ward transform of $T_\pi Z$, i.e. $TM\simeq \tau_\ast\eta^\ast T_\pi Z$. Similarly, the bundle $E$ is the Ward transform of $T_\pi Z\otimes \pi^\ast\sO(-1)$. 
\par
In \cite[\S 2]{BP1} we have identified the algebraic condition satisfied by the differential operator produced by the Ward transform from any $M$-uniform vector bundle $F$ on $Z$.
In our situation, we can state the results for $F=T_\pi Z(-1)$ as:
 \begin{proposition} The bundle $E$ is equipped with a first order differential operator $D:E\to E^\ast\otimes TM$ which satisfies
 $D(fs)=\sigma(df\otimes s)+fDs$, where $\sigma$ (the principal symbol of $D$) is the composition of the following two maps
 \begin{equation}\begin{CD} 
 T^\ast M\otimes E @> \alpha^\ast\otimes 1 >> E^\ast\otimes \cx^2\otimes E @> 1\otimes\alpha>>   E^\ast\otimes TM
 \end{CD}\label{phi}\end{equation}
 (where $\cx^2\simeq (\cx^2)^\ast$ as explained in Remark \ref{subtle}).\hfill $\Box$
 \label{D}\end{proposition}
 
 \begin{remark} On $M\backslash\Delta$\; $\sigma$ is invertible and $\sigma^{-1}\circ D$ is the standard hyperholomorphic connection on $E$, i.e. its tensor product with the standard flat connection on $\cx^2$ is the (holomorphic) Obata connection on $M\backslash\Delta$.
 \end{remark}
 \begin{remark} Given any first order differential operator $D:E\to F$ between (sections of) vector bundles on a manifold $M$, with symbol $\sigma:E\otimes T^\ast M \to F$, we can ``tensor" it with any connection $\nabla$ on a vector bundle $W$ over $M$:
 $$ (D\otimes_\sigma \nabla) (e\otimes w)=D(e)\otimes w + (\sigma\otimes 1)(e\otimes \nabla w).$$
 The symbol of this new operator is $\sigma\otimes 1$. We can do this for our operator $D$ and the flat connection on $\cx^2$. We obtain a differential operator $\tilde D:E\otimes \cx^2\to E^\ast\otimes TM\otimes \cx^2$ which extends the Obata connection. \label{tensor}\end{remark}
\begin{remark} The results claimed by Pantillie \cite{Pant} would imply that the Obata connection extends to a differential operator satisfying $\tilde D(fs)=\alpha^\ast(df)\otimes s+f\tilde Ds$, but we have trouble following his arguments (in particular the second last paragraph in the proof of his Theorem 2.1).
\end{remark}

We can view $\sigma^{-1}\circ D$ as a meromorphic connection on $E$, with polar set $\Delta$. Similarly the Obata connection on $M\backslash \Delta$ can be viewed as a meromorphic connection on $TM$ with polar set $\Delta$. It follows from Proposition \ref{D} that $\sigma^{-1}$ generally has a double pole along $\Delta$ and, hence, so does $\sigma^{-1}\circ D$. We shall now discuss conditions under which the pole becomes simple.
\par
Let $z=0$ be the local equation of $\Delta$. The meromorphic connection $\sigma^{-1}\circ D$ has a simple pole if $\lim_{z\to 0} z^2\sigma^{-1}\circ D=0$. Let us trivialise locally $E$, so that $\alpha$ is an endomorphism of the trivial bundle. We can then write $z=\det\alpha$, and owing to Proposition \ref{D}, we have:
$$ z^2\sigma^{-1}( De)=((\alpha^\ast)_{\rm adj}\otimes 1)(1\otimes \alpha_{\rm adj})(De),$$
where the subscript ``adj" denotes the classical adjoint. Thus, we can conclude:
\begin{lemma} The meromorphic connection $\sigma^{-1}\circ D$ has a simple pole along $\Delta$ provided that $De|_x\in E^\ast_x\otimes\im\alpha_x$ for any $x\in \Delta_{\rm reg}$ and any local section $e$ of $E$. If this is the case, then the residue of $\sigma^{-1}\circ D$ belongs to $\Ker\alpha^\ast\otimes \End E$.\hfill$\Box$\label{lemma}\end{lemma}
Returning to the description of a $2$-Kronecker structure given at the end of \S\ref{Kr}, observe that the subsheaf $TM[-\sV]$ is precisely the subsheaf $\im\alpha$, and therefore the condition of the last lemma is equivalent to the existence of a differential operator
$$ D^\prime:E\to E^\ast \otimes \cx^2\otimes E$$
such that $D=(1\otimes\alpha)\circ D^\prime$.

We shall now show that for a logarithmic hypercomplex structure (Definition \ref{log}) the condition of the above lemma is automatically satisfied.

\begin{proposition} Suppose that $\im\alpha_x=T_x\Delta$ for each $x\in \Delta_{\rm reg}$. Then the condition of Lemma \ref{lemma} is satisfied.\end{proposition}
\begin{proof} Proposition \ref{Tw} implies that points of $\Delta_{\rm reg}$ correspond to sections of $\pi:Z\to \oP^1$ contained in a codimension $1$ submanifold $Z^\prime$ of $Z$. 
The differential operator $D$ is obtained by the push-forward of the flat relative connection $\nabla_\eta$ on $\eta^\ast T_\pi Z(-1)$, i.e. of the exterior derivative in the vertical directions of the projection $\eta:M\times \oP^1\to Z$. It follows that, over $\Delta_{\rm reg}$, $D$ restricts to an operator $D^\prime$ defined in the same way as $D$, but with $Z$ replaced by $Z^\prime$. This means that $D^\prime$ takes values in $E\otimes T\Delta_{\rm reg}$.
\end{proof}

Recall that a meromorphic connection on a vector bundle $E$ is called logarithmic, if it has a simple pole along $\Delta=\{z=0\}$ and its residue is of the form $Adz$, where
$A\in\End E$. Thus, under the assumption of the last proposition,  $\sigma^{-1}\circ D$ is a logarithmic connection. 
\par
We finish with some remarks about the meromorphic Obata connection. As remarked in \ref{tensor}, we can tensor $D$ with the flat connection on $\cx^2$ to obtain an operator $\tilde D:E\otimes \cx^2\to E^\ast\otimes TM\otimes \cx^2$. The meromorphic Obata connection is a meromorphic connection on $TM$ given
by $(\sigma\otimes 1)^{-1}\circ \tilde D\circ \alpha^{-1}$, where $\sigma$ is the symbol of $D$. Thus, in general we can expect the Obata connection to have a third order pole along the divisor $\Delta$ ($\sigma^{-1}$ contributing two orders and $\alpha^{-1}$ another one). The next example shows that this is indeed the case.
\begin{example} Consider again the twistor space of the Calabi-Eguchi-Hanson gravitational instanton, described in Example \ref{CEH}. Choose a family of real sections containing the (real) jumping lines. The resulting metric is given in the complex coordinates corresponding to the complex structure $I$, namely $z=-\bar{a}$, $u=\ln \bar A^2$ by the formula \cite[(4.6)]{Hit1}:
$$ \gamma dzd\bar z+(du+\bar\delta dz)(d\bar u+\delta d\bar z),$$
where $\gamma$ and $\delta$ are certain functions of the coordinates and the fold $\Delta$ is given by $\gamma=0$. The Hermitian matrix of this metric is then
$$\begin{pmatrix} \gamma+\gamma^{-1}|\delta|^2 & \gamma^{-1}\bar\delta\\ \gamma^{-1}\delta & \gamma^{-1}\end{pmatrix},$$
and it follows that the Levi-Civita connection has poles of third order along $\Delta$.
\end{example}
It is interesting to observe that if the assumption of Lemma \ref{lemma} is satisfied, then the Obata connection still has a simple pole along $\Delta$ (rather than a second order one, as one could expect). Indeed, as noted above, the operator $D$ is then of the form $D=(1\otimes\alpha)\circ D^\prime$, where $ D^\prime:E\to E^\ast\otimes \cx^2  \otimes E$ has symbol $\alpha^\ast\otimes 1$.  It follows that the meromorphic Obata connection as an operator $TM\to T^\ast M\otimes TM$ is of the form 
$$ (1\otimes\alpha)\circ(\alpha^\ast\otimes 1)^{-1}\circ \widetilde{D^\prime}\circ \alpha^{-1}.$$
Since $(\alpha^\ast\otimes 1)^{-1}\circ \widetilde{D^\prime}$ is a meromorphic connection on $E\otimes\cx^2$ with a simple pole along $\Delta$, the meromorphic Obata connection also has a simple pole, as the conjugation by $\alpha$ does not increase the order of the pole.

In particular:
\begin{corollary} Suppose that the equivalent conditions of Proposition \ref{Tw} are satisfied. Then the holomorphic Obata connection on $M\backslash\Delta$ extends to a logarithmic connection on $M$.\hfill $\Box$
\end{corollary}

\end{document}